\documentclass[11pt]{amsart}
\usepackage{amssymb}
\usepackage{color}

\usepackage[hmargin={20mm,20mm},vmargin={15mm,20mm}]{geometry}
\usepackage{geometry}
\usepackage[dvips]{graphicx}
\usepackage{graphicx}
\usepackage{color}
\bibstyle{plain}
\newtheorem{theorem}{Theorem}[section]

\newtheorem{example}[theorem]{Example}

\newtheorem{lemma}[theorem]{Lemma}
\newtheorem{proposition}[theorem]{Proposition}
\newtheorem{corollary}[theorem]{Corollary}

\newtheorem{definition}[theorem]{Definition}
\newtheorem{remark}[theorem]{Remark}

\newcommand{\Xinf}{\ell_\infty(X)}

\newcommand{\dsupan}{d_{s, (a_n)}}
\newcommand{\dpan}{d_{p, (a_n)}}

\newcommand{\N}{\mathbb N}

\newcommand{\R}{\mathbb R}

\newcommand{\on}{\operatorname}

\begin{document}

%%%%%% TO BE ENTERED BY THE AUTHOR(S)
%%%
%%% ENTER TITLE
\title{A fixed point theorem for mappings on the $\ell_\infty$-sum of a metric space and its application}

%%% AUTHOR(S) FULL NAMES, AND EMAIL ADDRESSES

\author{Jacek Jachymski}
\address{Institute of Mathematics, \L\'od\'z University of Technology, W\'olcza\'nska 215, 93-005
\L\'od\'z, Poland}
\email {jacek.jachymski@p.lodz.pl}

\author{\L ukasz Ma\'slanka}
\address{Institute of Mathematics, \L\'od\'z University of Technology, W\'olcza\'nska 215, 93-005
\L\'od\'z, Poland}
\email {lukasz$\_$maslanka@interia.pl}

\author{Filip Strobin}
\address{Institute of Mathematics, \L\'od\'z University of Technology, W\'olcza\'nska 215, 93-005
\L\'od\'z, Poland}
\email {filip.strobin@p.lodz.pl}
%\title[On generalized iterated function systems defined on $\ell_\infty$-sum]{On generalized iterated function systems defined on $\ell_\infty$-sum of a metric space}

\begin{abstract}
The aim of this paper is to prove a counterpart of the Banach fixed point principle for mappings $f:\Xinf\to X$, where $X$ is a metric space and $\Xinf$ is the space of all bounded sequences of elements from~$X$. Our result generalizes the theorem obtained by Miculescu and Mihail in 2008, who proved a~counterpart of the Banach principle for mappings $f:X^m\to X$, where $X^m$ is the Cartesian product of $m$ copies of $X$. We also compare our result with a recent one due to Secelean, who obtained a weaker assertion under less restrictive assumptions.
We illustrate our result with several examples and give an application.
\end{abstract}

\maketitle

%%%%%% THIS PART MUST BE PLACED IMMEDIATELY AFTER THE \maketitle COMMAND
%%%%%% BACK TO ORIGINAL FOOTNOTES

%%%%%%

\section{Introduction}
%Let $(X,d)$ be a complete metric space. The classical Banach fixed point principle states that if $f:X\to X$ is a Banach contraction (that is, its Lipschitz constant $Lip(f)<1$), then there is a unique $x_*\in X$ such that
%\begin{equation}
%x_*=f(x_*)
%\end{equation}
%and for every $x_0\in X$, the sequence of iterates $(x_n)=(f^n(x_0))$ converges to $x_*$, with the rate of convergence estimated by
%\begin{equation}
%d(x_n,x_*)\leq\frac{Lip^n(f)}{1-Lip(f)}d(x_1,x_0)
%\end{equation}
%This famous result is of a particular interest and has numerous applications in various branches of mathematics.\\ 
If $(X,d)$ is a metric space and $m\in\N$, then by $X^m$ we denote the Cartesian product of $m$ copies of $X$. We~endow $X^m$ with the maximum metric:
\begin{equation*}
d_m((x_0,...,x_{m-1}),(y_0,...,y_{m-1})):= \max\{d(x_0,y_0),...,d(x_{m-1},y_{m-1})\}.
\end{equation*}
Miculescu and Mihail in \cite{MM} and \cite{M} obtained an interesting generalization of the Banach principle for mappings defined on $X^m$. Namely, they proved the following
\begin{theorem}\label{newf1}
Assume that $(X,d)$ is a complete metric space and $g:X^m\to X$ is such that the Lipschitz constant $Lip(g)<1$. Then there exists a unique point $x_*\in X$ such that $g(x_*,...,x_*)=x_*$. Moreover, for every $x_0,...,x_{m-1}\in X$, the sequence $(x_k)$ defined by
\begin{equation}\label{proceduraiteracyjna}
x_{m+k}=g(x_{k+m-1},...,x_{k}),\;\;k\geq 0,
\end{equation}
converges to $x_*$. 
\end{theorem}
{A point $x_*\in X$ which satisfies the equality $g(x_*,...,x_*)=x_*$ is called \emph{a generalized fixed point of $g$}}. \\
An interesting study of such fixed points can also be found in the paper \cite{CP} of Professor Ljubomir B. \'Ciri\'c and S.B. Pre\v si\'c.\\ 
Theorem \ref{newf1} gave a background for a version of the Hutchinson--Barnsley fractals theory for such mappings defined on finite Cartesian products -- see the above mentioned papers and the references therein. Also, note that the above theorem can be extended to mappings which satisfy weaker contractive conditions -- see, e.g., \cite{Ser} and \cite{SS}.

The next step was done by Secelean \cite{S}. Denote by $\ell_\infty(X)$ the $\ell_\infty$-sum of a metric space $X$, that is, the set of all bounded sequences of elements of $X$:
$$\Xinf:=\{(x_k)\subset X: (x_k) \textrm{ is bounded}\}.$$
Endow $\Xinf$ with the supremum metric:
\begin{equation}\label{se0}
d_{s}((x_n),(y_n)):=\sup\{d(x_n,y_n):n\in\N^*\},
\end{equation}
where $\N^*:=\{0,1,2,...\}$ (throughout the paper we enumerate sequences by nonnegative integers).
\begin{remark}\emph{
Let us notice that the notion of the $\Xinf$-sum of a family of spaces originates from functional analysis; see, e.g., \cite[p. xii]{LT}.}
\end{remark}
\begin{remark}\emph{
{It is also worth to observe that} if $X$ is bounded, then $\Xinf$ is exactly the product of countably many copies of $X$, that is, $\ell_\infty(X)=\prod_{k=0}^\infty X$. {On the other hand, if $X$ is unbounded, then $\ell_\infty(X)$ is a proper subspace of $\prod_{k=0}^\infty X$.}}
\end{remark}
%Given a metric space $(X,d)$, let $\Xinf$ be endowed with the metric $d_{s,1}$. We recall that

If $f:\Xinf \to X$, then we define $f_{s}: X\to X$ by
\begin{equation}\label{se0,5}
f_{s}(x):=f(x,x,...),\;\;x\in X.
\end{equation}{
A point $x_*\in X$ is called \emph{a generalized fixed point of $f$}, if $x_*$ is a fixed point of $f_s$, i.e., if $x_*$ satisfies:
\begin{equation*}
f(x_*,x_*,...)=x_*.
\end{equation*}}
%Fixed points of $f_s$ will be called \emph{generalized fixed points of $f$}. 
Secelean \cite[Theorem 3.1]{S} proved the following fixed point theorem:% (here $$, $x\in X$):
\begin{theorem}\label{secelean}
Assume that $X$ is a complete metric space and $f:\Xinf \to X$ is such that $Lip(f)<1$. Then there exists a unique generalized fixed point $x_*$ of $f$.
Moreover, for every $x=(x_n)\in \Xinf$, the sequence $(y_k)$ defined by
\begin{equation}\label{ciagseceleana}
y_k:=f\left(f_s^{k}(x_0),f_s^{k}(x_1),f_s^{k}(x_2),...\right),\;k\geq 0,
\end{equation} 
converges to $x_*$. More precisely, for every $k\in\N^*$,
$$
d(x_*,y_k)\leq \frac{\on{Lip}(f)^{k+1}}{1-\on{Lip}(f)}\sup\{d(f_s(x_i),x_i):i\in\N^*\}.
$$
\end{theorem}
\begin{remark}\emph{In fact, Secelean formulated his result in a more general way. Firstly, he considered also weaker contractive conditions and secondly, he studied also mappings defined on a finite product of spaces. However, the idea of dealing with weaker contractive conditions is relatively similar (but much more technically complicated), and also we will not be interested in the case of finite products here.}
\end{remark}
\begin{remark}\emph{
Theorem \ref{secelean} can be viewed as a generalization of the Banach fixed point theorem or Theorem~\ref{newf1}. However, it seems that the iteration procedure (\ref{ciagseceleana}) is not a very natural counterpart of (\ref{proceduraiteracyjna}). It~is~rather closer to iterating map $f_s$. %In fact, the existence of the fixed point $x_*$ follows from the Banach fixed point theorem used for $f_s$. %Moreover, the Theorem~\ref{secelean} can be partially proved with use of Theorem~\ref{filip7}.
}\end{remark}

We are going to show that under more restrictive (yet still natural) contractive conditions, we can obtain a stronger thesis. In particular, our result will imply the whole Theorem \ref{newf1}. Also, we will present examples that our assumptions are essential for the thesis and, in particular, that Theorem \ref{secelean} is too weak to obtain our assertion.\\
Finally, we will present {an application}.
%We are going to extend Theorem \ref{newf1} - instead of finite Cartesian product, we will consider spaces of all bounded sequences.
%The  paper is organized as follows.
%In the next section we will introduce some basic notations and initial results. Then in Section 3 we will prove the main result of the paper, and in Section 4 we will present some examples. Finally, in the last section we will show an application of our main result.

\section{Other metrics on $\ell_\infty(X)$}
\subsection{Metrics $d_{s,(a_n)}$ and $d_{p,(a_n)}$}
%We put $\N:=\{1,2,3,...\}$ and $\N^*:=\{0,1,2,...\}$. For the convenience, we will consider mostly sequences enumerated by numbers from $\N^*$. Thus $(x_n)=(x_n)_{n\in\N^*}$, if not stated otherwise.\\
%$$\textrm{Zakladamy, ze } \N^* = \{n\in\Z: n \geq 0\}.$$
Let $(X,d)$ be a metric space. We start with defining other metrics on space $\ell_\infty(X)$.
%For a metric space $(X, d)$, define $$\Xinf:=\{(x_n)\subset\X: (x_n) \textrm{ is bounded}\}$$
If $(a_n)$ is a sequence of reals, then set:% tending to $0$, then set
$$\dsupan(x,y) := \sup\{a_nd(x_n, y_n):n\in\N^*\} \;\;\textrm{ for any } x=(x_n), y=(y_n)\in \Xinf$$
and, if additionally {$a_n\geq 0$, $n\in\N^*$}, and $p\in[1,\infty)$, then set:
%Also, if additionally $\sum_{n=0}^{\infty} a_n < \infty$ and $p\in[1,\infty)$, then we put
$$\dpan(x,y) := \left( \sum_{n=0}^{\infty} a_nd^p(x_n, y_n) \right)^{1/p} \textrm{ for any } x=(x_n), y=(y_n)\in \Xinf. $$
%and if $p\in(0,1)$, then finally, we set
%$$\dinf(x,y) := \sum_{n=0}^{\infty} a_nd(x_n, y_n)^p  \textrm{ for any } x=(x_n), y=(y_n)\in \Xinf $$
It turns out that under natural assumptions on a sequence $(a_n)$, functions $d_{s,(a_n)}$ and $d_{p,(a_n)}$ are metrics with good properties:

%MMMMMMMMMMMMMMMMMMMM

\begin{proposition}\label{ds_char}
Let $(X, d)$ be a metric space such that $X$ is not a singleton and $(a_n)$ be a sequence of reals. The~following statements are equivalent:
\begin{itemize}
\item[(i)] $d_{s, (a_n)}$ is a metric on $\Xinf$;
\item[(ii)] $a_n > 0$ for any $n\in\N^*$ and $(a_n) \in l_\infty$.
\end{itemize}
Moreover, if $(a_n)$ is as in (ii), then the convergence with respect to $d_{s, (a_n)}$ implies the convergence in the Tychonoff product topology {(when considering $\ell_\infty(X)$ as a subspace of $\prod_{k=0}^\infty X$)}.
\end{proposition}

\begin{proof}
$(i) \Rightarrow (ii)$: Suppose, on the contrary, that $a_p \leq 0$ for some $p\in\N^*$. By hypothesis, there exist $x,y\in X$ such that $x \neq y$. Define
$\textbf{x}:= (x, x, ...)$ and $\textbf{y}=(x, ..., x, y, x, ...)$, where the $p$-th coordinate of \textbf{y} is equal to $y$. Then
$$0 < d_{s, (a_n)}(\textbf{x}, \textbf{y}) = \max\{0, a_pd(x,y)\} = 0,$$
which yields a contradiction. Thus $a_n > 0 $ for any $n\in\N^*$. \\
We show that $(a_n) \in l_\infty$.  Take again $x,y\in X$ with $x \neq y$ and define $\textbf{x}:= (x, x, ...)$ and $\textbf{y}=(y,y, ...)$.
Then $d_{s, (a_n)}(\textbf{x}, \textbf{y}) = \sup_{n\in\N^*} a_n d(x,y) = d(x,y) \sup_{n\in\N^*} a_n$, 
so 
$\sup_{n\in\N^*} a_n  = \frac{d_{s, (a_n)}(\textbf{x}, \textbf{y})}{d(x,y)} < \infty$. Thus $(a_n)$ is bounded.

The proof of $(ii) \Rightarrow (i)$ is standard and we leave it to the reader. 

Now assume that $d_{s, (a_n)}(x^k, {x}) \to 0$, where $x^k = (x^k_i)_{i\in\N^*}$ and ${x} = (x_i)_{i\in\N^*}$. Then for any $i\in\N^*$, 
$$0 \leq a_i d(x^k_i, x_i) \leq d_{s, (a_n)}(x^k, \textbf{x}),$$
which implies that $\lim_{k\to \infty} d(x^k_i, x_i) = 0$, i.e., $(x^k)$ converges to ${x}$ in the Tychonoff topology.
\end{proof}

\begin{proposition}\label{ds_top}
Let $(X, d)$ be a metric space such that $X$ is not a singleton and $(a_n)$ be a bounded sequence of positive reals. Let $\tau_{_T}$ denote the Tychonoff product topology on $\Xinf$ and $\tau_{d_{s, (a_n)}}$ be the topology induced by metric $d_{s, (a_n)}$. The~following statements are equivalent:
\begin{itemize}
\item[(i)] $\tau_{_T} = \tau_{d_{s, (a_n)}}$;
\item[(ii)] $(a_n) \in c_0$ and $(X,d)$ is bounded.
\end{itemize}
\end{proposition}

\begin{proof}
$(i) \Rightarrow (ii)$: Suppose, on the contrary, that $(a_n) \notin c_0$. Then there exist $\varepsilon_0 > 0$ and a subsequence $(a_{n_j})$ such that $a_{n_j} \geq \varepsilon_0$ for any $j\in\N^*$. Take $x,y\in X$ with $x \neq y$, and define $\textbf{x} = (x, x, ...)$ and $x^k = (x^k_i)_{i\in\N^*}$, where 
 \begin{equation*}x^k_i :=\left\{\begin{array}{ccc} x&\mbox{if}&i\leq k,\\y&\mbox{if}&i>k.\end{array}\right.\end{equation*}% \textrm{ for $k \leq n$ \; \; and \; \; } x^n_k := y \textrm{ for $k > n$}.$$
Clearly, $(x^k)$ converges to $\textbf{x}$ in $(X, \tau_{_T})$, so by (i), $d_{s, (a_n)} (x^k, \textbf{x}) \to 0$. On the other hand,
$$d_{s, (a_n)}(x^k, \textbf{x}) \geq \sup_{j\in\N^*} a_{n_j} d(x^k_{n_j}, x) \geq \varepsilon_0 d(x,y),$$
so letting $k$ tend to $\infty$, we obtain $0 \geq \varepsilon_0 d(x,y) > 0$, a contradiction. Thus $(a_n) \in c_0$. \\
Now, suppose that $(X, d)$ is unbounded. Then there exists a sequence $(x_k)$ such that $d(x_k, x_0) > \frac{1}{a_{k}}$ for any $k\in\N$. Set $\textbf{x} := (x_0, x_0, ...)$ and $x^k := (x^k_i)_{i\in\N^*}$, where 
\begin{equation*}x^k_i :=\left\{\begin{array}{ccc} x_0 &\mbox{if}& i\leq k,\\x_k&\mbox{if}&i>k.\end{array}\right.\end{equation*}% \textrm{ for $k \leq n$ \; \; and \; \; } x^n_k := x_n \textrm{ for $k > n$}.$$
Then $(x^k)$ converges to $\textbf{x}$ in $(X, \tau_{_T})$, so by (i), $d_{s, (a_n)} (x^k, \textbf{x}) \to 0$. However, if $k\geq 1$, then
$$d_{s, (a_n)} (x^k, \textbf{x}) = \sup_{j \geq k+1} a_j d(x_j, x_0) \geq a_{k+1} d(x_{k+1}, x_0) > 1,$$
which yields a contradiction.

$(ii) \Rightarrow (i)$: By the last part of Proposition~\ref{ds_char}, it suffices to show that the convergence in $(X, \tau_{_T})$ implies the convergence with respect to $d_{s, (a_n)}$. Assume that $x^k \stackrel{\tau_{_T}}{\to} {x}$, where $x^k = (x^k_i)_{i\in\N^*}$ and ${x} = (x_i)_{i\in\N^*}$. That means $\lim_{k\to\infty} d(x^k_i, x_i) = 0$ for any $i\in\N^*$. Fix $\varepsilon>0$. Since $a_n \to 0$, there is $p \in \N^*$ such that for $i>p$, $a_i < \frac{\varepsilon}{\on{diam}X}$. Then $a_id(x^k_i, x_i) < \varepsilon$ for $i>p$ and $k\in\N^*$. Since $\lim_{k\to\infty} d(x^k_i, x_i) = 0$ for $i=0,1, ..., p$, there is $j\in\N^*$ such that for $k\geq j$ and $i=0,...,p$, $d(x^k_i, x_i) < \frac{\varepsilon}{a_i}$. Then for $k \geq j$, $d_{s, (a_n)}(x^k, {x}) \leq \varepsilon$. Thus we get that $d_{s, (a_n)} (x^k ,{x}) \to 0$.
\end{proof}

Using a similar argument as in the proofs of Propositions~\ref{ds_char} and ~\ref{ds_top}, it is possible to prove the following two results for metrics $d_{p, (a_n)}$.

\begin{proposition}\label{dp_char}
Let $(X, d)$ be a metric space such that $X$ is not a singleton, $(a_n)$ be a sequence of {nonnegative} reals and $p\in [1,\infty)$. The following statements are equivalent:
\begin{itemize}
\item[(i)] $d_{p, (a_n)}$ is a metric on $\Xinf$;
\item[(ii)] $a_n > 0$ for any $n\in\N^*$ and $(a_n) \in l_1$.
\end{itemize}
Moreover, if $(a_n)$ is as in (ii), then the convergence with respect to $d_{p, (a_n)}$ implies the convergence in the Tychonoff product topology.
\end{proposition}

\begin{proposition}\label{dp_top}
Let $(X, d)$ be a metric space such that $X$ is not a singleton and $(a_n)$ be a sequence of positive reals such that $(a_n)\in l_1$. The following statements are equivalent:
\begin{itemize}
\item[(i)] $\tau_{_T} = \tau_{d_{p, (a_n)}}$;
\item[(ii)] $(X,d)$ is bounded.
\end{itemize}
\end{proposition}

%It can easily be seen that $\dsupan$ and $\dpan$, $p\in[1,\infty)$, are metrics.\\% (for $p\in(0,1)$ it follows from the inequality $(x+y)^p\leq x^p+y^p$, valid for all $x,y\geq 0$).\\
{In what follows, when writing $d_{s,(a_n)}$ (or $d_{p,(a_n)}$) we automatically assume that $(a_n)$ is chosen so that $d_{s,(a_n)}$ (or $d_{p,(a_n)}$) is a metric.}\\% we will always assume that $(X,d)$ is a fixed metric space and $(a_n)$ are chosen such that  $\dsupan$ and $\dpan$ are metrics.\\%,\;p\in[1,\infty)$ are defined as above for some sequences $(a_n)$.\\
A natural question arises whether these metrics are complete if $d$ is so. Clearly, if $a_n=1$ for all $n\in\N^*$, then $d_{s,(a_n)}$ is exactly the metric $d_s$ considered by Secelean, so it is complete. Also, if $\inf\{a_n:n\in\N^*\}>0$, then the metrics $d_{s,(a_n)}$ and $d_s$ are Lipschitz equivalent, hence $d_{s,(a_n)}$ is also  complete.\\
The following example shows that the answer can be negative if $a_n\to 0$.
\begin{example}\emph{
Let $(X,d):=(\R,|\cdot|)$ and for every $k\in\N^*$, let $x^k:=(0,1,..,k,0,0,...)$. Then:\\
- $\dsupan(x^k,x^{k+1})=(k+1)a_{k+1}$, so if $\sum (k+1)a_{k+1}<\infty$, then $(x^k)$ is Cauchy in $\dsupan$;\\
- if $p\geq 1$, then $\dpan(x^k,x^{k+1})=(k+1)a_{k+1}^{1/p}$, so if $\sum (k+1)a_{k+1}^{1/p}<\infty$, then $(x^k)$ is Cauchy in $\dpan$.\\
%- if $p<1$, then $\dpan(x_n,x_{n+1})=(n+1)^pa_n$, so if $\sum (n+1)^pa_n<\infty$, then $(x_n)$ is Cauchy in $\dpan$.\\
On the other hand, $(x^k)$ cannot be convergent since, by Propositions \ref{ds_char} and \ref{dp_char}, convergence in any of metrics $\dsupan,\dpan$ implies the convergence of each coordinate.}
\end{example}

%As a direct corollary of Propositions \ref{} and \ref{}, we get

\begin{corollary}\label{newf2}Assume that $a_n\to 0$.\\
(1) If $(X,d)$ is bounded and complete, then $\Xinf$ is complete with respect to any of metrics $\dsupan,\dpan$.\\
%Let $Y$ be nonempty, closed and bounded subsets of a metric space $(X,d)$. If $(X,d)$ is complete, then $\Y_\infty$ is complete with respect to any of metrics $\dsupan$ and $\dpan$.\\% restricted to $\Y:=Y_0\times Y_1\times Y_2\times...$ are complete.\\
(2) If $(X,d)$ is complete and $(x^k)=((x^k_i)_{i\in\N^*})$ is a Cauchy sequence in $\Xinf$ (with respect to any of metrics $\dsupan,\dpan$) such that the set $\{x^k_i: i,k\in\N^*\}$ is bounded in $X$, then $(x^k)$ is convergent to $x=(x_i)$, where $x_i=\lim_{k\to\infty}x^k_i$, $i\in\N^*$.\\
(3) If $x^k=(x^k_i)_{i\in\N^*}$, $k\in\N^*$, and $x=(x_i)$ are elements of $\Xinf$ such that the set $\{x^k_i: i,k\in\N^*\}$ is bounded, then {$x^k\overset{d'}\to x$ iff $x^k\to x$ with respect to the Tychonoff topology on $\Xinf$,} where $d'$ is any of metrics $\dsupan,\dpan$.\\
(4) If $(X,d)$ is compact, then $\Xinf$ is compact with respect to any of metrics $\dsupan,\dpan$.
\end{corollary}  
\begin{proof} 
(1). If $(x^k)=((x^k_i)_{i\in\N^*})$ is a Cauchy sequence in $\Xinf$, then each $(x^k_i)_{k\in\N^*}$ is Cauchy in $(X,d)$, hence convergent to some $x_i\in X$. Then by Propositions \ref{ds_top} and \ref{dp_top}, $x^k\to (x_i)$ with respect to any of metrics $d_{s,(a_n)},d_{p,(a_n)}$.\\
 (2) follows from (1) used for the subspace $\ell_\infty(Y)\subset \Xinf$, where $Y:=\overline{\{x^k_i: i,k\in\N^*\}}$.\\
(3) follows from Propositions \ref{ds_top} and \ref{dp_top} and (1), used for $Y:=\overline{\{x^k_i: i,k\in\N^*\}}$.\\
(4) is a direct consequence of Propositions \ref{ds_top} and \ref{dp_top}.
\end{proof}
\begin{remark}\emph{
It is worth to remark that the definitions of metrics $d_{s,(a_n)}$ and $d_{p,(a_n)}$ base on the same ideas as definitions of weighted $L_p$-sum of spaces considered in functional analysis (see for example \cite{G}). However, our setting is strictly metric.}
\end{remark}

\subsection{Particular versions of metrics $d_{s,(a_n)}$ and $d_{p,(a_n)}$: metrics $d_{s,q}$ and $d_{p,q}$}

From now on we will assume that $(a_n)$ is a geometric sequence $(q^n)$ for $q\in (0,1]$. As we will show, the obtained {results in such a case} imply corresponding results for the general case of $(a_n)$.\\ {For $q\in (0,1]$, denote $d_{s,q}:=d_{s,(q^n)}$, that is,
\begin{equation*}d_{s, q}(x,y) := \sup\{q^n d(x_n, y_n):n\in\N^*\} \;\;\textrm{ for any } x=(x_n), y=(y_n)\in \Xinf.
\end{equation*}
By Proposition \ref{ds_char}, $d_{s,q}$ is a metric. Observe that in this notation, the supremum metric $d_s$ is exactly the metric $d_{s,1}$.\\
If additionally $q<1$ and $p\in[1,\infty)$, denote $d_{p,q}:=d_{p,(q^n)}$, that is,
\begin{equation*}d_{p, q}(x,y) := \left( \sum_{n=0}^{\infty} q^n d^p(x_n, y_n) \right)^{1/p} \textrm{ for any } x=(x_n), y=(y_n)\in \Xinf. \end{equation*}
By Proposition \ref{dp_char}, $d_{p,q}$ is a metric.\\}
The following result shows some connections between $d_{s,q}$ and $d_{p,q}$.

\begin{proposition}\label{porownanie}
In the above frame, { assume that $q<1$ and $p\geq 1$. Then the following statements hold:}
\begin{itemize}
\item[(i)] $d_{s,q}\leq d_{p,q^p}$;
\item[(ii)] if $q\leq q'\leq 1$, then $d_{s,q}\leq d_{s,q'}$;
\item[(iii)] if $q^{{1}/{p}}<q'\leq 1$, then $d_{p,q} \leq \left(1-\frac{q}{(q')^p}\right)^{-1/p} d_{s, q'}$;
%\item[(iv)] If $q<1$, then $d_{p,q} \leq $ for $q' > q^{1/p}$.
\item[(iv)] for every $x,y\in\Xinf$, $\lim_{p'\to\infty} d_{p',q}(x,y) = d_{s,1}(x,y)$.%, where the convergence is uniform.
\end{itemize}
\end{proposition}

\begin{proof}
Let $x = (x_n), y = (y_n) \in \Xinf$.\\
We prove (i). Since $q^n \to 0$ and $(d(x_n, y_n))$ is bounded, we have for some $k_0\in\N^*$: 
$$d_{s,q}(x,y) = \sup_{n\in\N^*} q^n d(x_n, y_n) = q^{k_0} d(x_{k_0}, y_{k_0}) = \left(q^{pk_0} d^p(x_{k_0}, y_{k_0}) \right)^{1/p} \leq \left(\sum_{n=0}^\infty (q^p)^n d^p(x_n, y_n) \right)^{1/p} = d_{p, q^p}(x,y).$$
(ii) follows from the fact that for any $n\in\N^*$, $q^n d(x_n, y_n) \leq (q')^n d(x_n, y_n)$.\\
%(iii) is a consequence of
%$$d_{p,q}(x,y) = \left(\sum_{n=0}^\infty q^n d^p(x_n, y_n) \right)^{1/p} \leq \left(\sum_{n=0}^\infty q^n \right)^{1/p} d_{s,1}(x,y) = \left(\frac{1}{1-q}\right)^{1/p} d_{s,1}(x,y).$$
(iii) follows from
$$d_{p,q}(x,y) = \left(\sum_{n=0}^\infty q^n d^p(x_n, y_n) \right)^{1/p} = \left(\sum_{n=0}^\infty \frac{q^n}{(q')^{pn}} \big((q')^n d(x_n, y_n)\big)^p \right)^{1/p}  $$
$$\leq \left(\sum_{n=0}^\infty \left(\frac{q}{(q')^p}\right)^n \right)^{1/p} d_{s,q'}(x,y) = \left(1-\frac{q}{(q')^p}\right)^{-1/p} d_{s,q'}(x,y).$$
We prove (iv).  Let $\varepsilon>0$. Then there exists $k_0\in\N^*$ such that:
$$d_{s,1}(x,y) \leq  d(x_{k_0}, y_{k_0})+\varepsilon = \frac{1}{q^{k_0/p}} \left(q^{k_0} d^p(x_{k_0}, y_{k_0}) \right)^{1/p} +\varepsilon\leq \frac{1}{q^{k_0/p}} d_{p, q} (x,y)+\varepsilon.$$
Hence, 
$q^{k_0/p} d_{s,1}(x,y) \leq d_{p,q}(x,y) +\varepsilon$ and %\leq  (1-q)^{-1/p} d_{s,1}(x,y)+\varepsilon.$$
therefore:
$d_{s,1}(x,y) \leq \liminf_{p\to\infty} d_{p,q}(x,y)+\varepsilon$. Since $\varepsilon>0$ was arbitrary,
we have $d_{s,1}(x,y) \leq \liminf_{p\to\infty} d_{p,q}(x,y)$. On the other hand, (iii) (with $q'=1$) implies that
$ \limsup_{p\to\infty} d_{p,q}(x,y) \leq d_{s,1}(x,y).$
Thus we arrive to the desired equality.
\end{proof}
{By the previous section, if $X$ is bounded, then all metrics $d_{s,q}$ and $d_{p,q}$ are equivalent (and generate the Tychonoff topology on $\ell_\infty(X)=\prod_{k=0}^\infty X$). 
In general, this is not the case. For example, $d_{s,q}$ and $d_{p,q^p}$ need not be equivalent (recall point (i) of the above proposition), as the next example shows:
\begin{example}\emph{
Let $q \in (0,1)$, $(\R, |\cdot|)$ be the Euclidean space and $p\geq 1$. For $k\in\N$, let $x^k=(x^k_i)_{i\in\N^*}$ be defined by
\begin{equation*}
x^k_i:=\left\{\begin{array}{ccc}\frac{1}{(k+1)^{1/p}q^i}&\mbox{if}&i\leq k,\\0&\mbox{if}&i>k.\end{array}\right.
\end{equation*}
Then $x^k\to(0)$, the zero sequence, with respect to $d_{s,q}$, but does not converge with respect to $d_{p,q^p}$. Indeed, for every $k\in\N^*$,
\begin{equation*}
d_{s,q}(x^k,(0))= \sup\{ q^n d(x^k_n, 0):n\in\N^*\}=\frac{1}{(k+1)^{1/p}},
\end{equation*}
but
\begin{equation}
d_{p,q^p}(x^k, (0)) = \left(\sum_{n=0}^\infty (q^p)^n d^p(x_n^k, 0)\right)^{1/p} =
\left(\sum_{n=0}^k \frac{1}{(k+1)}\right)^{1/p}=1.
\end{equation}
}
\end{example}}
\section{Main results}
%In this section we will prove the main result of the paper.
\subsection{Sequences of generalized iterates and a selfmap of $\Xinf$}
For any mapping $f:\Xinf\to X$, define $\tilde{f}:\Xinf\to\Xinf$ as follows:
\begin{equation*}
\tilde{f}((x_n)) = (f((x_n)), x_0, x_1, ...) \textrm{ for any } (x_n)\in\Xinf.
\end{equation*}
Now if $x:=(x_n)\in\Xinf$, then we set $\tilde{x}^0:=x$ and 
\begin{equation*}
x^1:=f(\tilde{x}^0)\;\;\mbox{and}\;\;\tilde{x}^1 := \tilde{f}(\tilde{x}^0).
\end{equation*} 
Assume that for some $k\in\N$, we defined $x^i\in X ,\;\mbox{and}\; \tilde{x}^i\in\Xinf$ for $i\in\{1, ...,k\}$. Then set 
\begin{equation*}
x^{k+1}:=f(\tilde{x}^k)\;\;\mbox{and}\;\;\tilde{x}^{k+1}:=\tilde{f}(\tilde{x}^k).
\end{equation*}
In this way we defined sequences $(x^k)\subset X$ and $(\tilde{x}^k)\subset\Xinf$. Observe that
for every $k\in\N$,
\begin{equation}\label{filip3}
\tilde{x}^{k}=(f(\tilde{x}^{k-1}), ..., f(\tilde{x}^0), x_0, x_1, ...)=(x^{k},...,x^1,x_0,x_1,...).
\end{equation}
%and, in particular, $\tilde{x}^{k}\in\Xinf$.\\
Clearly, $(\tilde{x}^k)$ is the sequence of {iterates} of $x=(x_n)\in\Xinf$ of mapping $\tilde{f}$. We will say that $(x^k)$ is \emph{the sequence of generalized iterates of function $f$ at $x$.}\\% sequence ${x}$ of function $f$}.\\
Recall that $x_*\in X$ is a generalized fixed point of $f$, if
\begin{equation*}
f(x_*,x_*,...)=x_*.
\end{equation*}
\begin{definition}\emph{
A generalized fixed point $x_*\in X$ of map $f:\Xinf\to X$ is called a }generalized contractive fixed point \emph{ (GCFP), if for every $x\in\Xinf$, the sequence $(x^k)$ of generalized iterates converges to $x_*$.}
\end{definition}
The above definition is a counterpart of the notion of a contractive fixed point of a selfmap of a metric space introduced by Leader and Hoyle \cite{LH}:\\ if $g:Y\to Y$, then the fixed point $y_*\in Y$ of $g$ is called \emph{a contractive fixed point} (CFP), if for every $y\in Y$, the sequence of iterates $(g^k(y))$ converges to $y_*$.\\
We will show that the existence of a GCFP of $f$ is strongly related to the existence of a CFP of $\tilde{f}$. We start with the lemma { which follows directly from (\ref{filip3}): 
\begin{lemma}\label{filip2}In the above frame let $x=(x_n)\in\Xinf$.% and $y=(y_n)\in\Xinf$.
\begin{itemize}
\item[(i)] If $\tilde{x}^k\to\mathbf{x}$ for some $\mathbf{x}\in\Xinf$ with respect to the Tychonoff topology, then $\mathbf{x}=(x,x,x,...)$ for some $x\in X$, and $x^k\to x$.
\item[(ii)] If $x^k\to x$ for some $x\in X$, then $\tilde{x}^k\to \mathbf{x}$ with respect to the Tychonoff topology, where $\mathbf{x}=(x,x,x,...)$.
\end{itemize}
\end{lemma}  }
We are ready to state the theorem:
\begin{theorem}\label{filip4}
In the above frame,
\begin{itemize}
\item[(i)] $f$ has a GCFP iff $\tilde{f}$ has a CFP { with respect to the Tychonoff topology on $\ell_\infty(X)$};
\item[(ii)] if $\mathbf{x}_*$ is a CFP of $\tilde{f}$ { with respect to the Tychonoff topology}, then $\mathbf{x}_*=(x_*,x_*,...)$, where $x_*$ is a GCFP of $f$.
\end{itemize}
\end{theorem}
\begin{proof}
Let $\mathbf{x}_*$ be a CFP of $\tilde{f}$. Then by Lemma \ref{filip2}(i), $\mathbf{x}_*=(x_*,x_*,...)$ for some $x_*\in X$ and $x^k\to x_*$ for every $x\in\Xinf$ (as $\tilde{x}^k\to \mathbf{x}_*$ by hypothesis). Also, 
\begin{equation*}
(f(x_*,x_*,...),x_*,x_*,....)=\tilde{f}(x_*,x_*,...)=(x_*,x_*,...),
\end{equation*}
so $x_*$ is a generalized fixed point of $f$, and in view of the above observations, it is a GCFP.  Conversely, if $x_*$ is a GCFP of ${f}$, then by Lemma \ref{filip2}(ii), $\tilde{x}^k\to \mathbf{x}_*$ for any $x\in\Xinf$, where $\mathbf{x}_*=(x_*,x_*,x_*,...)$. As $\mathbf{x}_*$ is obviously a fixed point of $\tilde{f}$, it is a CFP.
\end{proof}

{\begin{remark}\label{lastfil}\emph{It is worth to observe that the convergence of the sequence $(\tilde{x}^k)$ of iterates of $\tilde{f}$ with respect to the Tychonoff topology is equivalent to the convergence with respect to any of metrics $d_{p,q}$ and $d_{s,q}$ if $q<1$. Indeed, this follows from Corollary \ref{newf2} and (\ref{filip3}).}
\end{remark} 
}
%Thus the searching for a GCFP of $f$ reduces to investigating a CFP of $\tilde{f}$. The next lemma shows the relationships between Lipschitz constants of $f$ and $\tilde{f}$ with respect to considered metrics.% In the following, by $L_{s,q},\tilde{L}_{s,q},L_{p,q},\tilde{L}_{p,q}$ we denote the Lipschitz constants of $f$ and $\tilde{f}$, with respect to metrics $d_{s,q}$ and $d_{p,q}$.
\subsection{A fixed point theorem}
If $f:\Xinf\to X$, then let $L_{s,q}(f)$ be the Lipschitz constant of $f$ with respect to $d_{s,q}$ on $\Xinf$, and let $L_{p,q}(f)$ be the Lipschitz constant of $f$ with respect to $d_{p,q}$ on $\Xinf$. {Similarly, by $\tilde{L}_{s,q}(\tilde{f})$ and $\tilde{L}_{p,q}(\tilde{f})$ we denote the Lipschitz constants of corresponding map $\tilde{f}$.}
\begin{remark}\emph{
In this framework, Secelean's Theorem \ref{secelean} says that if $L_{s,1}(f)<1$, then $f$ admits a unique generalized fixed point, and for every $(x_n)\in\Xinf$, the sequence $(y_k)$ defined by (\ref{ciagseceleana}) converges to this fixed point.\\
Our main result says what happens if we assume contractive conditions with respect to $d_{p,q}$ or $d_{s,q}$ with $q<1$.}
\end{remark}
The next lemma shows the relationships between Lipschitz constants of $f$ and $\tilde{f}$ with respect to the considered metrics.
\begin{lemma}\label{filip5}
In the above frame, if $f:\ell_\infty(X)\to X$, then
\begin{itemize}
\item[(i)] $\tilde{L}_{s,q}(\tilde{f})\leq \max\{q, L_{s,q}(f)\}$, where $q\leq 1$;
\item[(ii)] $\tilde{L}_{p,q}(\tilde{f})\leq \left((L_{p,q}(f))^p+q\right)^{1/p}$, where $q<1$ and $p\geq 1$.
\end{itemize}%\\
%(iii) if $p<1$, then $\tilde{L}_p\leq \left(a_0L_p^p+M\right).$
\end{lemma}
\begin{proof}
Let $x=(x_n), y=(y_n)\in \Xinf$. We have:
$$d_{s,q}(\tilde{f}(x), \tilde{f}(y)) = \sup{ \left\{d(f(x), f(y)), qd(x_0, y_0), ..., q^n d(x_{n-1}, y_{n-1}), ... \right\} } $$
$$\leq \sup{ \left(\{L_{s,q}(f) d_{s,q}(x, y)\}\cup \left\{q^n d(x_{n-1}, y_{n-1}): n\in\N \right\} \right)} $$
$$\leq\max\{L_{s,q}(f) d_{s,q}(x, y), q d_{s,q}(x, y)\} = \max\{q, L_{s,q}(f)\} d_{s,q}(x, y),$$
so we get (i). If $p\geq 1$ and $q<1$, then 
$$d_{p,q}(\tilde{f}(x), \tilde{f}(y)) = \left(d^p(f(x), f(y)) + \sum_{n=1}^{\infty} q^n d^p(x_{n-1}, y_{n-1}) \right)^{1/p}  $$
$$\leq \left((L_{p,q}(f))^p d_{p,q}^p(x, y) + q \sum_{n=1}^{\infty} q^{n-1}d^p(x_{n-1}, y_{n-1}) \right)^{1/p} \leq \left((L_{p,q}(f))^p d_{p,q}^p(x,y)+q d_{p,q}^p(x, y) \right)^{1/p}$$
$$=\left((L_{p,q}(f))^p+q\right)^{1/p} d_{p,q}(x,y),$$
so we get (ii). %Similarly we deal with the case $p<1$. 
\end{proof}

We are ready to state the main result of the paper.
\begin{theorem}\label{filip7}
Assume that $(X,d)$ is a complete metric space, and $f:\Xinf\to X$ satisfies one of the following conditions:
\begin{equation*}
(Q)\;\;\;\;\;\;\;L_{s,q}(f)<1\;\;\;\;\mbox{for some }q\in(0,1);
\end{equation*}
\begin{equation*}
(P)\;\;\;\;\;\;\;L_{p,q}(f)<(1-q)^{1/p}\;\;\;\;\mbox{for some $q\in(0,1)$ and }p\in[1,\infty).
\end{equation*} 
Then $f$ has a GCFP.\\
Moreover, if $x_*\in X$ is a GCFP of $f$ and $x\in\Xinf$, then
\begin{itemize}
\item[(i)] if $L_{s,q}(f)<1$ for some $q<1$, {it holds}
\begin{equation}
\label{supestimation}d(x^k,x_*)\leq L_{s,q}(f)\frac{\max\{L_{s,q}(f),q\}^{k-1}}{1-\max\{L_{s,q}(f),q\}} d_{s,q}(\tilde{x}^1,\tilde{x}^0);%\;\;\mbox{provided }L'<1
\end{equation}
\item[(ii)] if $L_{p,q}(f)<(1-q)^{1/p}$, {it holds}
\begin{equation}\label{pestimation}
d(x^k,x_*)\leq L_{p,q}(f)\frac{((L_{p,q}(f))^p+q)^{\frac{k-1}{p}}}{1-((L_{p,q}(f))^p+q)^{\frac{1}{p}}} d_{p,q}(\tilde{x}^1,\tilde{x}^0).%,\;\;\mbox{ provided }L''<1%\;\mbox{ and }p\geq 1$$
\end{equation}
\end{itemize}
%and
%$$d(x^k,x_*)\leq L_p\frac{L''}{1-(L'')}\dsup(\tilde{x}^1,\tilde{x}^0),\;\;\mbox{ provided }L''<1\;\mbox{ and }p< 1$$
\end{theorem}
\begin{proof}
We first deal with the case $L_{p,q}(f)<(1-q)^{1/p}$. %We deal with the second case in an analogous way. 
By Lemma \ref{filip5}, the Lipschitz constant of $\tilde{f}$ satisfies 
$$\tilde{L}_{p,q}(\tilde{f})\leq (L_{p,q}(f)^p+q)^{1/p}<((1-q)+q)^{1/p}=1,$$ so $\tilde{f}$ is a Banach contraction with respect to $d_{p,q}$ on $\Xinf$. Now take any $x=(x_n)\in\Xinf$. Then %the sequence $(\tilde{x}^k)$ is Cauchy. Indeed, 
for every $m>k$, we have
$$
d_{p,q}(\tilde{x}^k,\tilde{x}^m)\leq d_{p,q}(\tilde{x}^k,\tilde{x}^{k+1})+...+d_{p,q}(\tilde{x}^{m-1},\tilde{x}^{m})\leq d_{p,q}\left(\tilde{f}^k(\tilde{x}^0),\tilde{f}^k(\tilde{x}^1)\right)+...+d_{p,q}\left(\tilde{f}^{m-1}(\tilde{x}^0),\tilde{f}^{m-1}(\tilde{x}^1)\right)\
$$
$$
\leq (\tilde{L}_{p,q}(\tilde{f}))^k(d_{p,q}(\tilde{x}^0,\tilde{x}^1)+...+(\tilde{L}_{p,q})^{m-k-1}d_{p,q}(\tilde{x}^0,\tilde{x}^1))\leq \frac{(\tilde{L}_{p,q}(\tilde{f}))^k}{1-\tilde{L}_{p,q}(\tilde{f})}d_{p,q}(\tilde{x}^0,\tilde{x}^1),%\leq \frac{(L'')^{k/p}}{1-(L'')^{1/p}}d_{p,q}(\tilde{x}^0,\tilde{x}^1)
$$
{which means that $(\tilde{x}^k)$ is a Cauchy sequence with respect to $d_{p,q}$. Moreover
$$
d(x^{k+1},x^{m+1})=d(f(\tilde{x}^{k}),f(\tilde{x}^m))\leq L_{p,q}(f)d_{p,q}(\tilde{x}^{k},\tilde{x}^{m}) \leq L_{p,q}(f)\frac{(\tilde{L}_{p,q}(\tilde{f}))^k}{1-\tilde{L}_{p,q}(\tilde{f})}d_{p,q}(\tilde{x}^0,\tilde{x}^1),
$$
which means that $(x^k)$ is a Cauchy sequence in $X$. Hence the set $\{x^k: k\in\N\} \cup \{x_n: n\in\N^*\}$ is bounded and by Corollary \ref{newf2}(2), $\tilde{x}^k\overset{d_{p,q}}{\to} \mathbf{x}$ for some $\mathbf{x}\in\Xinf$. Since $\tilde{f}$ is continuous with respect to $d_{p,q}$, the point $\mathbf{x}$ is a fixed point of $\tilde{f}$, which must be unique as $\tilde{L}_{p,q}(\tilde{f})<1$. Hence $\mathbf{x}$ is a CFP of $\tilde{f}$ (with respect to Tychonoff topology -- see Remark \ref{lastfil}), and by Theorem \ref{filip4}, $\mathbf{x}=(x_*,x_*,...)$, where $x_*$ is a CGFP of $f$. Moreover, by the above computations, for every $x\in X$ and $m>k$, we have}
\begin{equation}\label{filip6}
d(x^{k+1},x^{m+1}) \leq L_{p,q}(f)\frac{(\tilde{L}_{p,q}(\tilde{f}))^k}{1-\tilde{L}_{p,q}(\tilde{f})}d_{p,q}(\tilde{x}^0,\tilde{x}^1)\leq L_{p,q}(f)\frac{((L_{p,q}(f))^p+q)^{\frac{k}{p}}}{1-((L_{p,q}(f))^p+q)^{\frac{1}{p}}} d_{p,q}(\tilde{x}^1,\tilde{x}^0). 
\end{equation} 
{Letting $m\to\infty$, we get}
$$
d(x^k,x_*)\leq  L_{p,q}(f)\frac{((L_{p,q}(f))^p+q)^{\frac{k-1}{p}}}{1-((L_{p,q}(f))^p+q)^{\frac{1}{p}}} d_{p,q}(\tilde{x}^1,\tilde{x}^0)
$$
for all $k\in\N$.\\
To get the assertion for assumption (Q) we could follow the same lines. However, as we will see in a moment, conditions (Q) and (P) are equivalent.
\end{proof}
\begin{remark}\label{abc3}\emph{
As was announced, a bit surprisingly, conditions (P) and (Q) are equivalent. In fact, each of them is also equivalent to a particular version of (P). More precisely, for every $f:\Xinf\to X$, the following conditions are equivalent:
\begin{itemize}
\item[(i)] $f$ satisfies (Q), that is, for some $q\in(0,1)$, $L_{s,q}(f)<1$;
\item[(ii)] $f$ satisfies (P), that is, for some $q\in(0,1)$ and $p\in[1,\infty)$, $L_{p,q}(f)<(1-q)^{1/p}$;
\item[(iii)] for every $q\in(0,1)$ there exists $p\in[1,\infty)$ such that $L_{p,q}(f)<(1-q)^{1/p}$.
\end{itemize} 
%$(i)\Rightarrow(ii)$.
% Assume first that $L_{s,q}(f)<1$ for some $q\in(0,1)$. By Proposition~\ref{porownanie}(ii), for every $p\in[1,\infty)$,
%$$d(f(x), f(y)) \leq L_{s,q}(f) d_{s,q}(x,y) \leq L_{s,q}(f) d_{p, q^p}(x,y),$$
%hence $L_{p,q^p}(f) \leq L_{s,q} <1$. Moreover, 
%
%hence for some $p\geq 1$, $L_{p,q^p}(f)<(1-q^p)^{1/p}$ and $(P)$ is satisfied.\\
{We first prove $(i)\Rightarrow(iii)$. Assume that $L_{s,q}(f)<1$ for some $q\in(0,1)$, and choose any $q_0\in(0,1)$. Observe that
$$
\lim_{p\to\infty}(1-q_0)^{1/p}=1,
$$
so we can take $p\in[1,\infty)$ so that $L_{s,q}(f)<(1-q_0)^{1/p}$ and also $q^p\leq q_0$. Then let $q'\in[q,1)$ be such that $(q')^p=q_0$. By Proposition \ref{porownanie}(i),(ii) we have for all $x,y\in \Xinf$,
$$
d(f(x),f(y))\leq L_{s,q}(f)d_{s,q}(x,y)\leq L_{s,q}(f)d_{s,q'}(x,y)\leq L_{s,q}(f)d_{p,(q')^p}(x,y)=L_{s,q}(f)d_{p,q_0}(x,y).
$$
Hence $L_{p,q_0}(f)\leq L_{s,q}(f)<(1-q_0)^{1/p}$. Thus we get $(iii)$.\\
Implication
$(iii)\Rightarrow(ii)$ is obvious.\\
Finally, we prove $(ii)\Rightarrow(i)$.
Assume that $L_{p,q}(f)<(1-q)^{1/p}$ for some $q\in(0,1)$ and $p\in[1,\infty)$. By~Proposition~\ref{porownanie}(iii) for $q' > q^{1/p}$:
$$L_{s,q'}(f) \leq \frac{L_{p,q}(f)}{(1-\frac{q}{(q')^p})^{1/p}}.$$
Taking the limit with $q'\to 1$, we get
$$\lim_{q'\to1}L_{s,q'}(f) \leq \frac{L_{p,q}(f)}{(1-{q})^{1/p}}<\frac{(1-q)^{1/p}}{(1-q)^{1/p}}=1,$$
which means that $L_{s,q'}(f)<1$ for some $q'<1$ and we get (i).}}
\end{remark}

\begin{remark}\label{newremarka}\emph{{
In view of (iii) from Remark \ref{abc3}, we see that for every $q_0\in(0,1)$, condition (P) is equivalent to
$$
(P_{q_0})\;\;\;L_{p,q_{0}}(f)<(1-q_0)^{1/p}\;\;\;\mbox{for some }p\in[1,\infty).
$$
Later we will see that we cannot restrict to arbitrary $q_0$ in $(Q)$, and also we cannot restrict to arbitrary $p_0\in[1,\infty)$ in (P).}}
\end{remark}

\begin{remark}\emph{
Since (P) and (Q) are equivalent, formally it is enough to consider just one of them ((Q) seems to be more natural). On the other hand, the theory works properly for both types of metrics. In particular, we get natural estimations (\ref{supestimation}) and (\ref{pestimation}).
}\end{remark}
\begin{remark}\emph{By Proposition \ref{porownanie}(ii) we see that for any $q\in(0,1)$, $L_{s,1}(f)\leq L_{s,q}(f)$. Hence if (Q) (or, equivalently, (P)) is satisfied, then also the assumptions of Theorem~\ref{secelean} are  satisfied. {(In fact, at the end of~\cite{S}, Secelean considered the metric $d_{1,\frac{1}{2}}$ and observed these relationships.)}.
It turns out that the converse is not true, as the next example shows.}
\end{remark}
\begin{example}\label{abc2}\emph{
Let $X:=[0,1]$ and $f((x_n)) := \frac{1}{2} \sup\{ x_n:n\in\N^*\}$. Then clearly $L_{s,1}(f) =\frac{1}{2}< 1$, so the assumptions of Theorem~\ref{secelean} are satisfied and $x_*=0$ is a generalized fixed point of $f$. However, if $x=(x_n)\in\ell_\infty([0,1])$ is such that for some $i\in\N^*$, $x_i:=\delta>0$, then for any $k\in\N$,
$x^k \geq  \frac{1}{2}\delta.$
In particular, the sequence of generalized iterations $(x^k)$ does not converge to $x_*=0$ and $f$ has no GCFP.}
\end{example}
\begin{remark}\emph{
Theorem~\ref{filip7} can be formulated in a more general way. Namely, assume that $(X,d)$ is complete and a sequence $(a_n)$ of positive reals satisfies $M:=\sup_{n\in\N} \frac{a_{n}}{a_{n-1}} < 1$, and let $f:\ell_\infty(X)\to X$ be such that one of the following conditions holds:
\begin{itemize}
\item[(i)] $a_0L_{s,(a_n)}(f)<1$;
\item[(ii)] $L_{p,(a_n)}(f)<\left(\frac{1-M}{a_0}\right)^{1/p}$ for some $p\in[1,\infty)$,
\end{itemize}
where $L_{s,(a_n)}(f)$ and $L_{p,(a_n)}(f)$ are Lipschitz constants of $f$ with respect to metrics $d_{s,(a_n)}$ and $d_{p,(a_n)}$, respectively. Then $f$ has a GCFP.\\
However, this assertion follows directly from Theorem \ref{filip7}. Indeed, for any $x=(x_n), y=(y_n) \in \Xinf$ we have:
$$d_{p, (a_n)}(x,y)  = \left( \sum_{n=0}^\infty a_n d^p(x_n, y_n) \right)^{1/p} \leq \left( \sum_{n=0}^\infty a_0M^n d^p(x_n, y_n) \right)^{1/p} = a_0^{1/p} d_{p, M}(x,y)$$
and therefore $L_{p, M}(f) \leq a_0^{1/p}L_{p, (a_n)}(f)$, so (ii) implies (P). Similarly, we can see that (i) implies (Q).}
\end{remark}
%\begin{remark}\emph{Finally, let us remark that in the proof of Theorem \ref{filip7} we proved that $\tilde{f}$ is a Banach contraction on $\ell_\infty(X)$, but we could not use the Banach theorem since $\Xinf$ need not be complete. [WYRZUCIC?] On the other hand, we believe that we can extend the space $\Xinf$ to the space of all sequences $x=(x_n)$ such that for some fixed sequence $z=(z_n)$, the $d_{s,q}(x,z)<\infty$ (or $d_{p,q}(x,z)<\infty$). Such spaces should be complete and thanks to it we probably could use the Banach theorem. We leave these considerations to the future work.}
%\end{remark}

In the last section we are going to use Theorem \ref{filip7} to prove Theorem \ref{newf1}. However, now we will show another connection between mappings on finite Cartesian products and mappings defined on spaces of sequences:

\begin{theorem}\label{approx}
Assume that $(X,d)$ is a complete metric space and let $f:\Xinf\to X$ satisfy (Q) (or, equivalently,~(P)) for $q\in(0,1)$. Choose any $x\in X$ and for any $n\in\N$, define $f_n:X^n \to X$ as follows:
\begin{equation}\label{finiteapprox}
\forall_{(x_0, ..., x_{n-1})\in X^n}\ f_n(x_0, ..., x_{n-1}) := f(x_0, ..., x_{n-1}, x, x, ...).
\end{equation}
Then for any $n\in\N$, $Lip(f_n)\leq L_{s,q}(f)$ (w.r.t. maximum metric $d_m$ on $X^n$) and the sequence $(x_*^n)$ of generalized fixed points of $f_n$s' (whose existence follows from Theorem \ref{newf1}) converges to $x_*$, a generalized fixed point of $f$. More precisely, for every $n\in\N$,
\begin{equation}\label{abc1}
d(x_*^n,x_*)\leq q^n\frac{L_{s,q}(f)}{1-L_{s,q}(f)}d(x_*,x).
\end{equation}
%where $q\in(0,1)$ is such that $L_{s,q}(f)<1$.
%Moreover, if by $x_n^k$ we denote the $k$-th iteration of $f_n$ of $(x_0, ..., x_{n-1})$, then the following approximation holds:
%$$d(x_n^k, x_*) \leq \frac{Lip^{[k/n]}(f_n)}{1-Lip(f_n)} d_{\max} \left((x_0, ..., x_{n-1}), (x_n^1, x_0, ..., x_{n-2})\right) + M^{n/p} \left(\frac{a_0L_p^p}{1-(a_0L_p^p+M)}\right)^{1/p} d(x, x_*).$$
\end{theorem}

\begin{proof}
%By Lemma~\ref{dpdsup} we can focus only on the case when $L''<1$. Let $n\in\N, x\in X$ and $x_0, ..., x_{n-1}, y_0, ..., y_{n-1} \in X$. Then
Assume that $L_{s,q}(f)<1$ for some $q\in(0,1)$. For every $n\in\N$, we have
$$d\left(f_n(x_0,...,x_{n-1}), f_n(y_0,...,y_{n-1})\right) = d\left(f(x_0, ..., x_{n-1}, x, x, ...), f(y_0, ..., y_{n-1}, x, x, ...)\right)  $$
$$
\leq L_{s,q}(f)d_{s,q}((x_0,...,x_{n-1},x,...),(y_0,...,y_{n-1},x,...))$$ $$=L_{s,q}(f)\max\{q^kd(x_k,y_k):k=0,...,n-1\}\leq L_{s,q}(f)d_m((x_0,...,x_{n-1}),(y_0,...,y_{n-1})).
$$
Hence $Lip(f_n)\leq L_{s,q}(f)<1$
%$$\leq L_{p,q}\left(\sum_{i=0}^{n-1} q^i d^p(x_{i}, y_{i}) \right)^{1/p} \leq L_{p,q} \left(\sum_{i=0}^{n-1} q^i\right)^{1/p} \max_{i=0,...,n-1}d(x_i, y_i) \leq \left(\frac{(L_{p,q})^p}{1-q} \right)^{1/p} \max_{i=0,...,n-1}d(x_i, y_i).$$
%Observe that since $(L_{p,q})^p + q < 1$, then also $\frac{(L_{p,q})^p}{1-q} < 1$, hence $Lip(f_n)<1$ 
and the assumptions of Theorem \ref{newf1} are fulfilled. Thus $f_n$ has a fixed point $x_*^n \in X$. Then we have
$$d(x_*^n, x_*) = d(f(x_*^n, ..., x_*^n, x, ...), f(x_*, ..., x_*, x_*, ...)) \leq L_{s,q}(f) d_{s,q}((x_*^n, ..., x_*^n, x,  ...), (x_*, ..., x_*, x_*,  ...)) $$
$$= L_{s,q}(f) \max\{d(x_*^n,x_*),q^nd(x_*,x)\}\leq L_{s,q}(f) (d(x_*^n,x_*)+q^nd(x_*,x)).
$$
%$$d^p(x_*^n, x_*) = d^p(f(x_*^n, ..., x_*^n, x, x, ...), f(x_*, ..., x_*, x_*, x_*, ...)) \leq $$
%$$\leq (L_{p,q})^p \left(\sum_{i=0}^{n-1} q^i d^p(x_*^n, x_*) + \sum_{i=n}^\infty q^i d^p(x, x_*) \right) \leq \frac{(L_{p,q})^p}{1-q} d^p(x_*^n, x_*) + q^n \frac{(L_{p,q})^p}{1-q} d^p(x, x_*).$$
Hence 
$$
d(x_*^n,x_*)\leq q^n\frac{L_{s,q}(f)}{1-L_{s,q}(f)}d(x_*,x).
$$
%$d^p(x_*^n, x_*) \leq q^n \frac{(L_{p,q})^p}{1-q} \left(1-\frac{(L_{p,q})^p}{1-q}\right)^{-1} d^p(x,x_*) = q^n \frac{(L_{p,q})^p}{1-\left((L_{p,q})^p+q\right)} d^p(x, x_*)$, and therefore:
%$$d(x_*^n, x_*) \leq q^{n/p} \left(\frac{(L_{p,q})^p}{1-\left((L_{p,q})^p+q\right)}\right)^{1/p} d(x, x_*).$$
%Furthermore:
%$$d(x_n^k, x_*) \leq d(x_n^k, x_*^n) + d(x_*^n, x_*) \leq $$
%$$\leq \frac{Lip^{[{k/n}]}(f_n)}{1-Lip(f_n)} d_{\max} \left((x_0, ..., x_{n-1}), (x_n^1, x_0, ..., x_{n-2})\right) + q^{n/p} \left(\frac{(L_{p,q})^p}{1-((L_{p,q})^p+q)}\right)^{1/p} d(x, x_*).$$
\end{proof}

%\begin{remark}\emph{Note that in original formulation of Theorem \ref{newf1}, there is estimated the distance between the fixed point $x_*$ of $f:X^m\to X$ and $k$-the iteration $x_k$. Thus we can combine this estimation with (\ref{abc1}) and obtain more applicable estimation.}
%\end{remark}
Finally, we give an example which shows that the thesis of the above theorem need not hold under the assumption $L_{s,1}(f)<1$:
\begin{example}\emph{Consider function $f$ from Example \ref{abc2}. Take any $x>0$ and for every $n\in\N$, let $f_n:[0,1]^n\to [0,1]$ be defined by (\ref{finiteapprox}), i.e., 
$$
f_n(x_0,..,x_{n-1}):=f(x_0,...,x_{n-1},x,x,...).
$$
Then, clearly, $f_n(x_0,...,x_{n-1})\geq \frac{1}{2}x$ for every $(x_0,...,x_{n-1})$. In fact, $x^n_*=\frac{1}{2}x$, so $(x^n_*)$ does not converge to $x_*=0$.}
\end{example}
%\section{A note on Secelan's result}

\section{An example}
To illustrate the considered machinery, we will calculate Lipschitz constants $L_{p,q}(f)$ and $L_{s,q}(f)$ in the case of mappings $f:\ell_\infty(\R)\to\R$ of the form
\begin{equation}\label{filip8}
f(x)=\sum_{n\in\N^*}b_nx_n,\;\;\mbox{for }x=(x_n)\in\ell_\infty(\R)
\end{equation}
for some sequence $(b_n)$ of reals with $\sum_{n\in\N^*}|b_n|<\infty$. We will use these calculations in a discussion connected with Remark \ref{abc3}.
%At first, let us calculate the Lipschitz constants of such maps over the considered metrics. In the following, all symbols have the same meaning as earlier.
\begin{proposition}\label{filip9}
If $f:\ell_\infty(\R)\to\R$ is defined by (\ref{filip8}), then 
$$L_{s,q}(f)=\sum_{n=0}^\infty\frac{|b_n|}{q^n}$$ and if $q<1$, then $$L_{p,q}(f)=\left\{\begin{array}{ccc}\left(\sum_{n=0}^\infty \frac{|b_n|^{p/(p-1)}}{(q^n)^{1/(p-1)}} \right)^{(p-1)/p}&\mbox{if}&p>1,\\
\sup\left\{\frac{|b_n|}{q^n}:n\in\N^*\right\}&\mbox{if}&p=1.\end{array}\right.$$%, and $L_p=\infty$ for $p<1$. 
\end{proposition}
\begin{proof}
Let $q<1$. Set $I_1:=\sup_{n\in\N^*}\frac{|b_n|}{q^n}$. Then for every $x=(x_n),y=(y_n)\in\ell_\infty(\R)$, we have
$$d(f(x), f(y)) = \left|\sum_{n=0}^\infty b_nx_n - b_ny_n\right| \leq \sum_{n=0}^\infty |b_n|\left|x_n-y_n\right| $$ $$=\sum_{n=0}^\infty \frac{|b_n|}{q^n}q^n\left|x_n-y_n\right|\leq I_1\sum_{n=0}^\infty q^n\left|x_n-y_n\right|=I_1d_{1,q}(x,y).$$
Now assume that $I_1<\infty$, and let $\varepsilon>0$. Then there is $n_0$ such that $\frac{|b_{n_0}|}{q^{n_0}}\geq I_1-\varepsilon$. If $y_n=0$ for all $n\in\N^*$ and $x_n=0$ for $n\neq n_0$ and $x_{n_0}=1$, then
$$
|f(x)-f(y)|=|b_{n_0}x_{n_0}|=\frac{|b_{n_0}|}{q^{n_0}}q^{n_0}|x_{n_0}|=\frac{|b_{n_0}|}{q^{n_0}}d_{1,q}(x,y)\geq (I_1-\varepsilon)d_{1,q}(x,y).
$$
Hence $L_{1,q}=I_1$. In a similar way we can show that $L_{1,q}=I_1$ when $I_1=\infty$.\\
Now assume $p>1$. %and set $I_p:=\left(\sum_{n=0}^\infty \frac{|b_n|^{p/(p-1)}}{a_n^{1/(p-1)}} \right)^{(p-1)/p}$. 
Then for every $x=(x_n),y=(y_n)\in\ell_\infty(\R)$, we have by the H\"older inequality:
$$|f(x)-f(y)| = \left|\sum_{n=0}^\infty b_n(x_n-y_n)\right| \leq \sum_{n=0}^\infty \frac{|b_n|}{(q^n)^{1/p}}\cdot (q^n)^{1/p} \left|x_n-y_n\right|   $$
$$\leq \left(\sum_{n=0}^\infty \frac{|b_n|^{p/(p-1)}}{(q^n)^{1/(p-1)}} \right)^{(p-1)/p} \left( \sum_{n=0}^\infty q^n |x_n-y_n|^p \right)^{1/p} = \left(\sum_{n=0}^\infty \frac{|b_n|^{p/(p-1)}}{(q^n)^{1/(p-1)}} \right)^{(p-1)/p} d_{p,q}((x_n), (y_n)).$$
Observe that the first inequality is the equality if $b_n(x_n-y_n)\geq 0$ for all $n\in\N^*$.
Moreover, from the H\"older inequality we know that the second inequality is the equality iff the sequences $\left(\left(\frac{|b_n|}{(q^n)^{1/p}}\right)^{p/(p-1)}\right)$, $\left(\left((q^n)^{1/p}|x_n-y_n|\right)^p\right)$ are linearly dependent. For every $n\in\N^*$, let $y_n:=0$, and define%Now if $N\in\N$, then set $x_n:=0$ for $n>N$ and
\begin{equation}\label{postac x_n}
x_n :=\left\{\begin{array}{ccc} \on{sgn}(b_n)\left(\frac{|b_n|}{q^n}\right)^{1/(p-1)} &\mbox{if}&n\leq N,\\0&\mbox{if}&n>N,\end{array}\right.%\textrm{ for } n=0,1,...,N.
\end{equation}
where $\on{sgn}(\cdot)$ denotes the sign function. Then by previous observations, replacing $(b_n)$ by $(b_n')$ defined by $b_n':=b_n$ for $n\leq N$ and $b_n':=0$ for $n>N$, we have $$|f(x)-f(y)|=\left(\sum_{n=0}^N \frac{|b_n|^{p/(p-1)}}{(q^n)^{1/(p-1)}} \right)^{(p-1)/p}d_{p,q}(x,y).$$
Since $N$ was taken arbitrarily, we get $L_{p,q}(f)=\left(\sum_{n=0}^\infty \frac{|b_n|^{p/(p-1)}}{(q^n)^{1/(p-1)}} \right)^{(p-1)/p}$.\\
%Now assume $p<1$, and choose $n_0\in\N^*$ so that $b_{n_0}\neq 0$. If $y_n:=0$ for $n\in\N^*$, and $x_n=0$ if $n\neq n_0$, and $x_{n_0}:=x\neq 0$, then
%$$
%\frac{|f(x)-f(y)|}{d_p(x,y)}=\frac{|b_{n_0}||x|}{a_{n_0}|x|^p}=\frac{|b_{n_0}|}{a_{n_0}}|x|^{1-p}
%$$
%Since $\lim_{x\to \infty}{|x|^{1-p}}= \infty$, we see that $L_p=\infty$.\\
Finally, for any $q\leq 1$ and every $x=(x_n),y=(y_n)\in\ell_\infty(\R)$, we have
$$
|f(x)-f(y)|\leq\sum_{n=0}^\infty|b_n||x_n-y_n|=\sum_{n=0}^\infty\frac{|b_n|}{q^n}q^n|x_n-y_n|\leq\left(\sum_{n=0}^\infty \frac{|b_n|}{q^n}\right) d_{s,q}(x,y).
$$
Now let $y_n:=0$ for all $n\in\N^*$, fix any $N\in\N$, and define $x_n:=0$ for $n>N$ and for $n=0,...,N$, set $x_n:=\frac{\on{sgn}(b_n)}{q^n}$. Then
$$
|f(x)-f(y)|=\left|\sum_{n=0}^N b_nx_n\right|=\left|\sum_{n=0}^N\frac{b_n}{q^n}q^n\frac{\on{sgn}(b_n)}{q^n}\right|=\sum_{n=0}^N\frac{|b_n|}{q^n}\cdot 1=\left(\sum_{n=0}^N\frac{|b_n|}{q^n}\right) d_{s,q}(x,y).
$$
Since $N$ was arbitrary, we have $L_{s,q}(f)=\sum_{n\in\N^*}\frac{|b_n|}{q^n}$.
\end{proof}
{\begin{remark}\emph{It is very likely that the above result can be obtained from functional analysis machinery, since~$f$ is a linear map which is a sum of linear maps. However, we presented here the proof for the sake of completeness.}
\end{remark}}
%At first we show that the assumptions in Theorem \ref{filip7} are essential.
\begin{example}\emph{
We will consider functions $f:\ell_\infty(\R)\to\R$ of the form (\ref{filip8}) with different sequences $(b_n)$.\\
(1) Let $b_0=0$ and $b_n=b^n$, where $b\in(0,\frac{1}{2}]$ is fixed. By Proposition \ref{filip9},
 $$
L_{s,q}(f)=\sum_{n=1}^\infty\left(\frac{b}{q}\right)^n=\left\{\begin{array}{ccc}\frac{b}{q-b}&\mbox{if}&q>b,\\
\infty&\mbox{if}&q\leq b.\end{array}\right.
$$
Now if $b<1/2$, then $L_{s,q}(f)<1$ iff $q>2b$. This shows that in the formulation of condition (Q) we cannot restrict to some particular value $q_0$ (compare Remark~\ref{newremarka}).\\
If $b=1/2$, then $\sum_{n=1}^\infty b_n=1$, so every $x\in \R$ is a generalized fixed point of $f$. Also  $\lim_{q\to 1}L_{s,q}(f)=1$, which shows that in Theorem \ref{filip7} we cannot assume that $L_{s,q}(f)>1$.\\
(2) Let $b_n=0$ for $n\neq 1$ and $b_1=b$, where $b\in(0,1]$ is fixed. By Proposition \ref{filip9}, for every $p\in[1,\infty)$,
$$
L_{p,q}(f)=\frac{b}{q^{1/p}}.
$$
Now if $b<1$, then $L_{p,q}(f)<(1-q)^{1/p}$ iff $b<(q-q^2)^{1/p}$. In particular, we can choose $q\in(0,1)$ and $p\in[1,\infty)$ such that $L_{p,q}(f)<(q-q^2)^{1/p}$. However, if we fix $p_0\in[1,\infty)$, then 
$$\sup\left\{(q-q^2)^{1/p}: p\in[1,p_0], q\in (0,1)\right\}=\frac{1}{4^{1/p_0}}<1.$$ 
This shows that in the formulation of condition (P) we cannot restrict to some particular value of $p_0$.\\
If $b=1$, then every $x\in\R$ is a generalized fixed point of $f$. Also, for every $q\in(0,1)$, $\lim_{p\to\infty}\frac{L_{p,q}(f)}{(1-q)^{1/p}}=\lim_{p\to\infty}\frac{1}{(q-q^2)^{1/p}}=1$. This shows that in Theorem \ref{filip7} we cannot assume that $L_{p,q}(f)>(1-q)^{1/p}$.}
\end{example}

\section{Applications}
At first we show that Theorem \ref{filip7} implies Theorem \ref{newf1} and, in particular, the classical Banach fixed point theorem. Recall that by $X^m$ we denote the Cartesian product of $m$ copies of $X$ and we endow $X^m$ with the maximum metric
$$
d_m((x_0,...,x_{m-1}),(y_0,...,y_{m-1})):=\max\{d(x_0,y_0),...,d(x_{m-1},y_{m-1})\}.
$$
\begin{proof}(of Theorem \ref{newf1})
Choose $q\in(0,1)$ such that ${Lip(g)}<q^{m-1}$. % and $M=\sup_{n\in\N^*}\frac{a_{n+1}}{a_n}<1$.
Define $f:\Xinf\to X$ by $f(x_0,x_1,x_2,...):=g(x_0,...,x_{m-1})$. For every $x=(x_n),y=(y_n)\in\Xinf$, we have
$$
d(f(x),f(y))=d(g(x_0,...,x_{m-1}),g(y_0,...,y_{m-1}))\leq Lip(g)\max\{d(x_0,y_0),...,d(x_{m-1},y_{m-1})\}$$ $$\leq \frac{Lip(g)}{q^{m-1}}\max\{q^0d(x_0,y_0),...,q^{m-1}d(x_{m-1},y_{m-1})\}\leq\frac{Lip(g)}{q^{m-1}}d_{s,q}(x,y).
$$
Hence $L_{s,q}(f)\leq \frac{Lip(g)}{q^{m-1}}<1$, so mapping $f$ satisfies the assumptions of Theorem \ref{filip7}. It remains to observe that if $x_0,...,x_{m-1}\in X$, then sequence $(x_k)$ defined by (\ref{proceduraiteracyjna}), is the sequence of generalized iterates of $f$ at $x:=(x_{m-1},...,x_0,x_0,x_0,...)$.
\end{proof}

The second application of our result deals with a recursive procedure which $"$looks back$"$ at all previously defined elements.
\begin{example}\emph{
Fix $(b_n)\subset\R$, $c\in\R$, and consider the sequence $(x^k)$ defined by the following linear recursion:
$$\left\{ \begin{array}{ll}
& x^1 := c,\\%\mbox{ where }(x_n)\in\R_\infty\mbox{ is initially chosen}\\
%&\ldots \\
& x^k := c+b_0x^{k-1}+b_1x^{k-2}+...+b_{k-2}x^1, \;\;k\geq 2.
%&(x_n)\in \R_\infty.
\end{array} \right.$$
Then $(x^k)$ is the sequence of iterates of $y=(0,0,...)$, of the map $f(x):=\sum_{n=0}^\infty b_nx_n+c$. Thus if the assumptions of Theorem \ref{filip7} are satisfied (the Lipschitz constants can be calculated as in Proposition \ref{filip9}), then $x^k\to x_*$, where $x_*$ is the GCFP of $f$, that is
$$
x_* = f(x_*, x_*, ...) = \sum_{n=0}^\infty b_n x_*+c =\left(\sum_{n=0}^\infty b_n\right)x_*+c,
$$
which gives $x_*=\frac{c}{1-\left(\sum_{n=0}^\infty b_n\right)}$.\\
For example, assume that $b_n := \frac{1}{3\cdot 2^n},\; n\in\N^*$ and $c=1$. Setting $q=\frac{4}{5}$, we have that $L_{s,q}(f) = \frac{8}{9}$, so $f$ fulfills the assumptions of Theorem \ref{filip7}. Thus $x^k\to x_*=\frac{c}{1-\sum_{n\in\N^*}b_n}=3$. Moreover, by the second part of Theorem~\ref{filip7}, for every $k\in\N$,
$$|x^k-3|\leq L_{s,q}(f)\frac{\max\{L_{s,q}(f),q\}^{k-1}}{1-\max\{L_{s,q}(f),q\}}d_{s,q}(\tilde{x}^1,(0))=9\left(\frac{8}{9}\right)^k
$$
since $d_{s,q}(\tilde{x}^1,(0))=d_{s,q}((c,0,0,...),(0,0,...))=c=1$.}
\end{example}
{
\begin{remark}\emph{
As we have already mentioned, Secelean \cite{S} used his theorem to study the Hutchinson--Barnsley theory of fractals for maps defined on $\Xinf$. In our paper \cite{MS} we use results of this article to obtain an appropriate version of the Hutchinson--Barnsley theory in such setting. 
}\end{remark}
}
%%% ENTER REFERENCES IN THE FORM


\begin{thebibliography}{33}

\bibitem{CP} L.B. \'Ciri\'c, S.B. Pre\v si\'c, \emph{On Pre{s}i\'c type generalization of the Banach contraction mapping principle}, Acta Math. Univ. Comenian. (N.S.) 76 (2007), 143--147.

\bibitem{G} L. Grafakos, \emph{Classical and Modern Fourier Analysis}. Pearson Education, Inc., Upper Saddle River, NJ, 2004.

\bibitem{LH} S. Leader, S.L. Hoyle,
\emph{Contractive fixed points}, 
Fund. Math. 87 (1975), 93--108. 

\bibitem{LT} J. Lindenstrauss, L. Tzafriri, \emph{Classical Banach spaces. I. Sequence spaces}, Ergebnisse der Mathematik und ihrer Grenzgebiete, Vol. 92. Springer-Verlag, Berlin-New York, 1977.

\bibitem{MS} \L. Ma\'slanka, F. Strobin, \emph{On generalized iterated function systems defined on $\ell^\infty$-sum of a metric space}, in preparation.

\bibitem{MM} R. Miculescu, A. Mihail,  \emph{Applications of fixed
point theorems in the theory of generalized IFS}, Fixed Point Theory Appl.
Volume 2008, Article ID 312876, 11 pages doi:10.1155/2008/312876.

\bibitem{M} A. Mihail, \emph{Recurrent iterated function systems}, Rev.
Roumaine Math. Pures Appl., 53 (2008),  43--53.




\bibitem{S} N. Secelean, \emph{Generalized iterated function systems on
the space $l^\infty(X)$}, J. Math. Anal. Appl. 410 (2014), 847--858.

\bibitem{Ser}M.A. Serban, \emph{Fixed point theorems for operators on Cartesian product spaces and applications}, International Conference on Nonlinear Operators, Differential Equations and Applications (Cluj-Napoca, 2001), Semin. Fixed Point Theory Cluj-Napoca  3  (2002), 163--172.

\bibitem{SS} F. Strobin, J. Swaczyna, \emph{On a certain
generalisation of the iterated function system}, Bull. Aust. Math. Soc. 87
(2013), 37--54.
\end{thebibliography}
\end{document}